\documentclass[oneside,reqno,11pt]{amsart}

\usepackage{geometry}
\usepackage{mathtools}
\usepackage{amsmath}
\usepackage{amsthm}
\usepackage{amssymb}
\usepackage{tikz}
\usepackage{tikz-cd}
\usepackage[all,cmtip]{xy}
\usepackage{float}
\usepackage{indentfirst}
\usepackage[mathscr]{euscript}
\usepackage{stmaryrd}
\usepackage{hyperref}
\usepackage{csquotes}

\makeatletter
\def\blfootnote{\gdef\@thefnmark{}\@footnotetext}
\makeatother

\theoremstyle{plain}
\newtheorem{thm}{Theorem}[section]
\newtheorem{prop}[thm]{Proposition}
\newtheorem{lem}[thm]{Lemma}
\newtheorem{cor}[thm]{Corollary}
\newtheorem{conj}[thm]{Conjecture}

\theoremstyle{definition}
\newtheorem{dfn}[thm]{Definition}

\newtheorem*{ack}{Acknowledgements}

\theoremstyle{remark}
\newtheorem{rmk}[thm]{Remark}

\numberwithin{equation}{section}

\title{Computation of categorical entropy via spherical functors}
\author{Jongmyeong Kim}
\address{Center for Geometry and Physics, Institute for Basic Science (IBS), Pohang 37673, Republic of Korea}
\email{myeong@ibs.re.kr}

\begin{document}

\begin{abstract}
We study the relationship between the categorical entropy of the twist and cotwist functors along a spherical functor.
In particular, we prove the categorical entropy of the twist functor coincides with that of the cotwist functor if the essential image of the right adjoint functor of the spherical functor contains a split-generator.
We also see our results generalize the computations of the categorical entropy of spherical twists and $\mathbb{P}$-twists by Ouchi and Fan.
As an application, we apply our results to the Gromov--Yomdin type conjecture by Kikuta--Takahashi.
\end{abstract}

\maketitle


\section{Introduction}

\subsection{Motivation}

A {\em (topological) dynamical system} $(X,f)$ consists of a topological space $X$ and a continuous map $f : X \to X$.
One way to measure the complexity of a dynamical system $(X,f)$ is to investigate the asymptotic behavior of the iterations of the map $f$.
It gives rise to the notion of the {\em topological entropy} $h_\mathrm{top}(f) \in [0,\infty)$.

For a \textquote{nice enough} dynamical system, its topological entropy coincides with a linear algebraic quantity.
The following is known as the {\em Gromov--Yomdin theorem}.

\begin{thm}[\cite{Gro1},\cite{Gro2},\cite{Yom}]\label{thm1.1}
Let $X$ be a smooth projective variety over $\mathbb{C}$ and $f : X \to X$ be a surjective endomorphism.
Then
\begin{equation*}
h_\mathrm{top}(f) = \log\rho(f^*)
\end{equation*}
where $f^* : H^*(X,\mathbb{C}) \to H^*(X,\mathbb{C})$ is the induced automorphism on the cohomology and $\rho$ is the spectral radius, i.e., the largest absolute value of the eigenvalues.
\end{thm}

Similarly, let us consider a {\em categorical dynamical system} $(\mathcal{D},\Phi)$ which means a pair of a triangulated category $\mathcal{D}$ and an exact endofunctor $\Phi : \mathcal{D} \to \mathcal{D}$.
The {\em categorical entropy} $h_t(\Phi)$ (carrying a parameter $t \in \mathbb{R}$) of a categorical dynamical system $(\mathcal{D},\Phi)$ was introduced by Dimitrov--Haiden--Katzarkov--Kontsevich \cite{DHKK} as a categorical analogue of the topological entropy.

Thus, in view of the Gromov--Yomdin theorem, it is natural to expect an analogous formula to hold for the categorical entropy.

\begin{conj}[\cite{KT}, Conjecture 5.3]\label{conj1.2}
Let $X$ be a smooth projective variety over $\mathbb{C}$ and $\Phi : D^b(X) \to D^b(X)$ be an exact autoequivalence.
Then
\begin{equation*}
h_0(\Phi) = \log\rho([\Phi])
\end{equation*}
where $[\Phi] : \mathcal{N}(X) \otimes_\mathbb{Z} \mathbb{C} \to \mathcal{N}(X) \otimes_\mathbb{Z} \mathbb{C}$ is the induced automorphism on the numerical Grothendieck group (tensored with $\mathbb{C}$).
\end{conj}

Kikuta--Takahashi \cite{KT} showed that the derived pullback $\Phi = \mathbb{L}f^*$ where $f : X \to X$ is a surjective endomorphism satisfies Conjecture \ref{conj1.2}.
It, in particular, implies that $h_0(\mathbb{L}f^*) = h_\mathrm{top}(f)$ by Theorem \ref{thm1.1}.
Besides that, Conjecture \ref{conj1.2} has been verified for a variety of cases such as smooth projective curves \cite{Kik}, smooth projective varieties with ample (anti)canonical bundles \cite{KT}, abelian surfaces \cite{Yos}, spherical twists \cite{Ouc} and $\mathbb{P}$-twists \cite{Fan2} while there are also counterexamples \cite{Fan1},\cite{Ouc},\cite{Mat}.
Therefore, it is important to characterize exact autoequivalences which do or do not satisfy Conjecture \ref{conj1.2}.

In this paper, we study the relationship between the categorical entropy of the twist and cotwist functors along a spherical functor whose notion was introduced by Anno--Logvinenko \cite{AL2} as a generalization of that of a spherical object introduced by Seidel--Thomas \cite{ST}.
In particular, we will generalize the following results by Ouchi \cite{Ouc} and Fan \cite{Fan2}.

\begin{thm}[\cite{Ouc}, Theorem 3.1]\label{thm1.3}
Let $\mathcal{D}$ be the perfect derived category of a smooth proper dg algebra and $E$ be a $d$-spherical object of $\mathcal{D}$.
Denote by $T^\mathbb{S}_E$ the spherical twist along $E$.
Then
\begin{equation*}
(1-d)t \leq h_t(T^\mathbb{S}_E) \leq
\begin{cases}
0 & (\text{if } t \geq 0),\\
(1-d)t & (\text{if } t \leq 0).
\end{cases}
\end{equation*}
If moreover $E^\perp = \{F \in \mathrm{Ob}(\mathcal{D}) \,|\, \mathrm{Hom}_\mathcal{D}^*(E,F) = 0\} \neq 0$ then $h_t(T^\mathbb{S}_E) = 0$ for all $t \geq 0$.
\end{thm}

\begin{rmk}
Before Ouchi's work, Ikeda proved the same formula for the spherical twists along the simple modules of the Ginzburg dg algebras associated with acyclic quivers \cite[Proposition 4.5]{Ike}.
\end{rmk}

\begin{thm}[\cite{Fan2}, Theorem 3.1]\label{thm1.5}
Let $\mathcal{D}$ be the perfect derived category of a smooth proper dg algebra and $E$ be a $\mathbb{P}^d$-object of $\mathcal{D}$.
Denote by $T^\mathbb{P}_E$ the $\mathbb{P}$-twist along $E$.
Then
\begin{equation*}
-2dt \leq h_t(T^\mathbb{P}_E) \leq
\begin{cases}
0 & (\text{if } t \geq 0),\\
-2dt & (\text{if } t \leq 0).
\end{cases}
\end{equation*}
If moreover $E^\perp = \{F \in \mathrm{Ob}(\mathcal{D}) \,|\, \mathrm{Hom}_\mathcal{D}^*(E,F) = 0\} \neq 0$ then $h_t(T^\mathbb{P}_E) = 0$ for all $t \geq 0$.
\end{thm}

\subsection{Results}

The main observation is that, when we consider a $d$-spherical object (resp. $\mathbb{P}^d$-object) $E$ as a spherical functor $S$ in a standard way, its twist functor $T_S$ is isomorphic to the spherical (resp. $\mathbb{P}$-)twist along $E$ and its cotwist functor $C_S$ is isomorphic to the shift functor $[-1-d]$ (resp. $[-2-2d]$).
Since $h_t(C_S[2]) = (1-d)t$ (resp. $-2dt$), it is natural to expect that Theorems \ref{thm1.3} and \ref{thm1.5} still hold for the spherical twist along a general spherical functor after replacing $(1-d)t$ (resp. $-2dt$) in the theorems by $h_t(C_S[2])$.

It turns out this is the case under a certain mild assumption.

\begin{thm}[Upper bound]\label{thm1.6}
Let $S : \mathcal{C} \to \mathcal{D}$ be a spherical functor.
Assume that the essential image of its right adjoint functor $R : \mathcal{D} \to \mathcal{C}$ contains a split-generator of $\mathcal{C}$.\footnote{
In this case, the essential image of $R$ is sometimes said to be {\em dense} in $\mathcal{C}$, i.e., the smallest full triangulated subcategory containing the essential image of $R$ and closed under taking direct summand coincides with $\mathcal{C}$.
Also note that this condition is equivalent to the condition that the essential image of the left adjoint functor $L : \mathcal{D} \to \mathcal{C}$ of $S$ contains a split-generator of $\mathcal{C}$ by Definition \ref{dfn2.6} (4).
}
Then
\begin{equation*}
h_t(T_S) \leq
\begin{cases}
0 & (\text{if } h_t(C_S[2]) \leq 0),\\
h_t(C_S[2]) & (\text{if } h_t(C_S[2]) \geq 0).
\end{cases}
\end{equation*}
\end{thm}

\begin{thm}[Lower bound]\label{thm1.7}
Let $S : \mathcal{C} \to \mathcal{D}$ be a spherical functor with right adjoint functor $R$.
\begin{itemize}
\item[(1)] Assume that the essential image of $R : \mathcal{D} \to \mathcal{C}$ contains a split-generator of $\mathcal{C}$.
Then
\begin{equation*}
h_t(T_S) \geq h_t(C_S[2]).
\end{equation*}
\item[(2)] Assume that $\mathrm{Ker}\, SR \neq 0$.\footnote{
This is equivalent to the condition that $\mathrm{Ker}\, SL \neq 0$ by Definition \ref{dfn2.6} (3).
}
Then
\begin{equation*}
h_t(T_S) \geq 0.
\end{equation*}
\end{itemize}
\end{thm}

Combining Theorems \ref{thm1.6} and \ref{thm1.7}, we immediately obtain the following corollary.

\begin{cor}
Let $S : \mathcal{C} \to \mathcal{D}$ be a spherical functor with right adjoint functor $R$.
Assume that the essential image of $R : \mathcal{D} \to \mathcal{C}$ contains a split-generator of $\mathcal{C}$.
Then
\begin{equation*}
h_t(T_S) = h_t(C_S[2])
\end{equation*}
for all $t \in \mathbb{R}$ such that $h_t(C_S[2]) \geq 0$.
In particular, we have
\begin{equation*}
h_0(T_S) = h_0(C_S[2]).
\end{equation*}
\end{cor}

The above results give a way to compute (or estimate) the categorical entropy of the twist functor $T_S$ using that of the cotwist functor $C_S$ and vice versa.
These are particularly useful when the categorical entropy of one of the twist and cotwist functors is easy to compute while the other is not.
Since every exact autoequivalences can be written as the twist functor along a spherical functor (in various ways) \cite{Seg}, these results may provide an available tool to compute the categorical entropy.

As a simple application of the above results, we give a sufficient condition for a twist functor to satisfy Conjecture \ref{conj1.2}.
Firstly, the following proposition says that the {\em Gromov type inequality} for the cotwist functor can be transferred to that for the twist functor under a certain assumption.

\begin{prop}\label{prop1.9}
Let $\mathcal{C},\mathcal{D}$ be triangulated categories of finite type with finite rank numerical Grothendieck groups and $S : \mathcal{C} \to \mathcal{D}$ be a spherical functor.
Assume that the essential image of its right adjoint functor $R : \mathcal{D} \to \mathcal{C}$ contains a split-generator of $\mathcal{C}$.
Assume moreover that
\begin{equation*}
h_0(C_S) \leq \log\rho([C_S])
\end{equation*}
and that there exists an element $v \in \mathcal{N}(\mathcal{C}) \otimes_\mathbb{Z} \mathbb{C}$ such that $[S]v \neq 0$ and $[C_S]v = \lambda v$ where $|\lambda| = \rho([C_S])$.
Then
\begin{equation*}
h_0(T_S) \leq \log\rho([T_S]).
\end{equation*}
\end{prop}

The opposite inequality, so-called the {\em Yomdin type inequality}, is known to hold for the perfect derived category of a smooth proper dg algebra \cite[Theorem 2.13]{KST} (see Theorem \ref{thm4.1}).
Combining this fact and Proposition \ref{prop1.9}, we obtain the following corollary.

\begin{cor}\label{cor1.10}
Let $\mathcal{C}$ be a triangulated category of finite type with finite rank numerical Grothendieck group, $\mathcal{D}$ be the perfect derived category of a smooth proper dg algebra and $S : \mathcal{C} \to \mathcal{D}$ be a spherical functor.
Assume that the essential image of its right adjoint functor $R : \mathcal{D} \to \mathcal{C}$ contains a split-generator of $\mathcal{C}$.
Assume moreover that
\begin{equation*}
h_0(C_S) \leq \log\rho([C_S])
\end{equation*}
and that there exists an element $v \in \mathcal{N}(\mathcal{C}) \otimes_\mathbb{Z} \mathbb{C}$ such that $[S]v \neq 0$ and $[C_S]v = \lambda v$ where $|\lambda| = \rho([C_S])$.
Then
\begin{equation*}
h_0(T_S) = \log\rho([T_S]).
\end{equation*}
\end{cor}

\subsection{Structure of the paper}

In Section \ref{sec2}, we briefly review the basic definitions and properties of categorical entropy and spherical functors.

Then, in Section \ref{sec3}, we prove Theorems \ref{thm1.6} and \ref{thm1.7} by generalizing the proofs of \cite[Theorem 3.1]{Ouc} (see Theorem \ref{thm1.3}) and \cite[Theorem 3.1]{Fan2} (see Theorem \ref{thm1.5}).
Bacause we consider general triangulated categories rather than restricting to the perfect derived categories of smooth proper dg algebras as in \cite{Ouc} and \cite{Fan2}, some technical but elementary lemmas have to be established.

In Section \ref{sec4}, we give a proof of Proposition \ref{prop1.9} using Theorems \ref{thm1.6} and \ref{thm1.7}.

Finally, in Section \ref{sec5}, we see some examples coming from various twist functor constructions: Seidel--Thomas' spherical twists \cite{ST}, Huybrechts--Thomas' $\mathbb{P}$-twists \cite{HT} and Anno--Logvinenko's orthogonally spherical twists \cite{AL1}.
In particular, we see Theorems \ref{thm1.3} and \ref{thm1.5} can be deduced from Theorems \ref{thm1.6} and \ref{thm1.7}.
We also give a criterion to check the technical condition in Theorems \ref{thm1.6} and \ref{thm1.7}.
\subsection{Conventions}

Throughout the paper, all triangulated categories are linear over a field $\mathbb{K}$.
Moreover, whenever we deal with spherical functors, every triangulated category $\mathcal{D}$ is assumed to admit a {\em dg enhancement} in the sense that there exists a (pretriangulated) dg category $\mathcal{A}$ whose homotopy category $H^0(\mathcal{A})$ is equivalent to $\mathcal{D}$ as triangulated categories.
Similarly, every exact functor $\Phi : H^0(\mathcal{A}) \to H^0(\mathcal{B})$ between such triangulated categories is assumed to admit a {\em dg lift}, i.e., there exists a dg functor $\tilde{\Phi} : \mathcal{A} \to \mathcal{B}$ which descends to $\Phi$.
Also we will assume that $\mathcal{D}$ has a split-generator whenever the categorical entropy of an endofunctor of $\mathcal{D}$ is considered.

\begin{ack}
This work was supported by the Institute for Basic Science (IBS-R003-D1).
\end{ack}

\section{Preliminaries}\label{sec2}

\subsection{Categorical entropy}

Let us begin with reviewing the notion of the categorical entropy introduced by Dimitrov--Haiden--Katzarkov--Kontsevich \cite{DHKK}.

\begin{dfn}[\cite{DHKK}, Definition 2.1]
Let $E,F$ be objects of a triangulated category $\mathcal{D}$.
The {\em categorical complexity} of $F$ with respect to $E$ is the function $\delta_t(E,F) : \mathbb{R} \to [0,\infty]$ in $t$ defined by
\begin{equation*}
\delta_t(E,F) = \inf
\left\{\sum_{i=1}^k e^{n_i t} \,\left|\,
\begin{tikzcd}[column sep=tiny]
0 \ar[rr] & & * \ar[dl] \ar[rr] & & * \ar[dl] & \cdots & * \ar[rr] & & F \oplus F' \ar[dl]\\
& E[n_1] \ar[ul,"+1"] & & E[n_2] \ar[ul,"+1"] & & & & E[n_k] \ar[ul,"+1"] &
\end{tikzcd}
\right.\right\}
\end{equation*}
if $F \not\cong 0$, and $\delta_t(E,F) = 0$ if $F \cong 0$.
\end{dfn}

The categorical complexity satisfies the following properties.

\begin{lem}[\cite{DHKK}, Proposition 2.3]\label{lem2.2}
Let $E,E',E''$ be objects of a triangulated category $\mathcal{D}$.
\begin{itemize}
\item[(1)] $\delta_t(E,E'') \leq \delta_t(E,E') \delta_t(E',E'')$.
\item[(2)] $\delta_t(E,E' \oplus E'') \leq \delta_t(E,E') + \delta_t(E,E'')$.
\item[(3)] $\delta_t(\Phi E,\Phi E') \leq \delta_t(E,E')$ for any exact functor $\Phi : \mathcal{D} \to \mathcal{D}'$.
\end{itemize}
\end{lem}

An object $G$ of a triangulated category $\mathcal{D}$ is called a {\em split-generator} if the smallest full triangulated subcategory containing $G$ and closed under taking direct summand coincides with $\mathcal{D}$.
The categorical entropy is then defined as follows.

\begin{dfn}[\cite{DHKK}, Definition 2.5]
Let $G$ be a split-generator of a triangulated category $\mathcal{D}$.
The {\em categorical entropy} of an exact endofunctor $\Phi : \mathcal{D} \to \mathcal{D}$ is the function $h_t(\Phi) : \mathbb{R} \to [-\infty,\infty)$ in $t$ defined by
\begin{equation*}
h_t(\Phi) = \lim_{n \to \infty} \frac{1}{n} \log \delta_t(G,\Phi^nG).
\end{equation*}
\end{dfn}

\begin{rmk}
The categorical entropy is well-defined (i.e., the limit exists in $[-\infty,\infty)$) and independent of the choice of a split-generator used to define it \cite[Lemma 2.6]{DHKK}.
Moreover, it can be computed as
\begin{equation*}
h_t(\Phi) = \lim_{n \to \infty} \frac{1}{n} \log \delta_t(G',\Phi^nG)
\end{equation*}
for any split-generators $G,G'$ of $\mathcal{D}$ \cite[Lemma 2.6]{Kik}.
\end{rmk}

\begin{lem}\label{lem2.5}
Let $\Phi : \mathcal{D} \to \mathcal{D}, \Phi' : \mathcal{D}' \to \mathcal{D}', \Psi : \mathcal{D} \to \mathcal{D}'$ be exact functors such that $\Psi \Phi \cong \Phi' \Psi$.
Assume that the essential image of $\Psi$ contains a split-generator of $\mathcal{D}'$.
Then
\begin{equation*}
h_t(\Phi) \geq h_t(\Phi').
\end{equation*}
\end{lem}

\begin{proof}
Let $G$ be a split-generator of $\mathcal{D}$ such that $\Psi G$ is a split-generator of $\mathcal{D}'$.
Then
\begin{align*}
h_t(\Phi)
&= \lim_{n \to \infty} \frac{1}{n} \log \delta_t(G,\Phi^nG)\\
&\geq \lim_{n \to \infty} \frac{1}{n} \log \delta_t(\Psi G,\Psi\Phi^nG)\\
&= \lim_{n \to \infty} \frac{1}{n} \log \delta_t(\Psi G,\Phi'^n\Psi G)\\
&= h_t(\Phi')
\end{align*}
where the second inequality follows from Lemma \ref{lem2.2} (3).
\end{proof}

\subsection{Spherical functors}

Now we review the notion of a spherical functor and the (co)twist functors along them introduced by Anno--Logvinenko \cite{AL2}.
We also show some of their properties which will be used to prove the main theorems.

\begin{dfn}[\cite{AL2}, Definition 5.2]\label{dfn2.6}
Let $\mathcal{C},\mathcal{D}$ be triangulated categories.
An exact functor $S : \mathcal{C} \to \mathcal{D}$ with right and left adjoint functors $R,L$ is called a {\em spherical functor} if it satisfies the following conditions:
\begin{itemize}
\item[(1)] The {\em twist functor} $T_S = \mathrm{Cone}(SR \overset{\varepsilon}{\to} \mathrm{Id}_\mathcal{D})$ is an exact autoequivalence of $\mathcal{D}$ where $\varepsilon : SR \to \mathrm{Id}_\mathcal{D}$ is the counit of the adjoint pair $S \dashv R$.
\item[(2)] The {\em cotwist functor} $C_S = \mathrm{Cone}(\mathrm{Id}_\mathcal{C} \overset{\eta}{\to} RS)[-1]$ is an exact autoequivalence of $\mathcal{C}$ where $\eta : \mathrm{Id}_\mathcal{C} \to RS$ is the unit of the adjoint pair $S \dashv R$.
\item[(3)] $R \cong LT_S[-1]$.
\item[(4)] $R \cong C_SL[1]$.
\end{itemize}
\end{dfn}

\begin{rmk}
In order to have functorial cones, we should take dg enhancements of triangulated categories $\mathcal{C},\mathcal{D}$ and dg lifts of exact functors $S,R,L$.
However, for the sake of simplicity, we will refrain from using the dg category language in this paper.
For a detailed argument, see \cite[Section 4]{AL2}.
\end{rmk}

\begin{thm}[\cite{AL2}, Theorem 5.1]
Any two out of the four conditions in Definition \ref{dfn2.6} imply the other two.
\end{thm}

\begin{lem}\label{lem2.9}
Let $S : \mathcal{C} \to \mathcal{D}$ be a spherical functor with right and left adjoint functors $R,L$.
Then we have the following isomorphisms:
\begin{itemize}
\item[(1)] $T_S^nS \cong S(C_S[2])^n$.
\item[(2)] $RT_S^n \cong (C_S[2])^nR$.
\item[(3)] $LT_S^n \cong (C_S[2])^nL$.
\item[(4)] $SRT_S^n \cong T_S^nSR$.
\item[(5)] $SLT_S^n \cong T_S^nSL$.
\item[(6)] $RS(C_S[2])^n \cong (C_S[2])^nRS$.
\item[(7)] $LS(C_S[2])^n \cong (C_S[2])^nLS$.
\end{itemize}
\end{lem}

\begin{proof}
(1)
It is enough to show the isomorphism for $n=1$.
Let $\varepsilon : SR \to \mathrm{Id}_\mathcal{D}$ be the counit and $\eta : \mathrm{Id}_\mathcal{C} \to RS$ be the unit of the adjoint pair $S \dashv R$.
By definition, they satisfy $\varepsilon S \circ S\eta = \mathrm{Id}_S$.
Then, by the octahedral axiom, we get
\begin{equation*}
\begin{tikzcd}
S \ar[r,"S\eta"] \ar[d,equal] & SRS \ar[d,"\varepsilon S"] \ar[r] & SC_S[1] \ar[r] \ar[d] & S[1] \ar[d,equal]\\
S \ar[r,"\mathrm{Id}_S",swap] & S \ar[r] \ar[d] & 0 \ar[r] \ar[d] & S[1]\\
& T_SS \ar[r,equal] \ar[d] & T_SS \ar[d] &\\
& SRS[1] \ar[r] & SC_S[2] &
\end{tikzcd}
\end{equation*}
and hence $T_SS \cong SC_S[2]$.

(2), (3)
These isomorphisms follow from Definition \ref{dfn2.6} (3), (4).

(4), (5), (6), (7)
These isomorphisms follow from (1), (2), (3).
\end{proof}

\section{Entropy of twist and cotwist functors}\label{sec3}

\subsection{Lower bound}

Let us first prove Theorem \ref{thm1.7}.

\begin{proof}[Proof of Theorem \ref{thm1.7}]
(1)
This immediately follows from Lemmas \ref{lem2.5} and \ref{lem2.9} (2).

(2)
Take an object $0 \neq E \in \mathrm{Ker}\, SR$.
Then, from the exact triangle $SRE \to E \to T_SE \to SRE[1]$, it follows that $T_SE \cong E$.
Thus, for a split-generator $G$ of $\mathcal{D}$, we have
\begin{align*}
h_t(T_S)
&= \lim_{n \to \infty} \frac{1}{n} \log \delta_t(G,T_S^n(G \oplus E))\\
&\geq \lim_{n \to \infty} \frac{1}{n} \log \delta_t(G,T_S^nE)\\
&= \lim_{n \to \infty} \frac{1}{n} \log \delta_t(G,E)\\
&= 0.
\end{align*}
This completes the proof.
\end{proof}

\subsection{Upper bound}

We say that an object $F$ of a triangulated category $\mathcal{D}$ admits a {\em cone decomposition} with components $(E_1,E_2,\dots,E_k)$ if there is a sequence of exact triangles in $\mathcal{D}$ of the form
\begin{equation*}
\begin{tikzcd}[column sep=tiny]
0 \ar[rr] & & * \ar[dl] \ar[rr] & & * \ar[dl] & \cdots & * \ar[rr] & & F \ar[dl]\\
& E_1 \ar[ul,"+1"] & & E_2 \ar[ul,"+1"] & & & & E_k. \ar[ul,"+1"] &
\end{tikzcd}
\end{equation*}

The following two lemmas follow from the well-known fact that {\em a cone decomposition can be \textquote{inserted} into another cone decomposition}: a cone decomposition of $F$ with components $(E_1,E_2,\dots,E_k)$ and a cone decomposition of $E_i$ with components $(E'_1,E'_2,\dots,E'_l)$ give rise to a cone decomposition of $F$ with components $(E_1,\dots,E_{i-1},E'_1,\dots,E'_l,E_{i+1},\dots,E_k)$.

\begin{lem}\label{lem3.1}
Let $E$ be an object and $F' \to F \to F'' \to F'[1]$ be an exact triangle of a triangulated category $\mathcal{D}$.
Then
\begin{equation*}
\delta_t(E,F) \leq \delta_t(E,F') + \delta_t(E,F'').
\end{equation*}
\end{lem}

\begin{proof}
Take $\varepsilon > 0$ and choose a cone decomposition of $F' \oplus \tilde{F}'$ with components
\begin{equation*}
(E[n_1],E[n_2],\dots,E[n_k])
\end{equation*}
and a cone decomposition of $F'' \oplus \tilde{F}''$ with components
\begin{equation*}
(E[m_1],E[m_2],\dots,E[m_l])
\end{equation*}
so that
\begin{equation*}
\sum_{i=1}^k e^{n_it} < \delta_t(E,F') + \varepsilon \quad\text{and}\quad \sum_{j=1}^l e^{m_jt} < \delta_t(E,F'') + \varepsilon.
\end{equation*}

Then, from the exact triangle $F' \oplus \tilde{F}' \to F \oplus \tilde{F}' \oplus \tilde{F}'' \to F'' \oplus \tilde{F}'' \to F' \oplus \tilde{F}'[1]$, we see that $F \oplus \tilde{F}' \oplus \tilde{F}''$ admits a cone decomposition with components
\begin{equation*}
(E[n_1],\dots,E[n_k],E[m_1],\dots,E[m_l]).
\end{equation*}
It shows that
\begin{align*}
\delta_t(E,F)
&\leq \sum_{i=1}^k e^{n_it} + \sum_{j=1}^l e^{m_jt}\\
&< \delta_t(E,F') + \delta_t(E,F'') + 2\varepsilon.
\end{align*}
Since $\varepsilon > 0$ is arbitrary, the assertion follows.
\end{proof}

\begin{lem}\label{lem3.2}
Let $F,G$ be objects and $E' \to E \to E'' \to E'[1]$ be an exact triangle of a triangulated category $\mathcal{D}$.
Then
\begin{equation*}
\delta_t(G,F) \leq (\delta_t(G,E') + \delta_t(G,E''))\delta_t(E,F).
\end{equation*}
\end{lem}

\begin{proof}
Take $\varepsilon > 0$ and choose a cone decomposition of $F \oplus \tilde{F}$ with components
\begin{equation*}
(E[n_1],E[n_2],\dots,E[n_k]),
\end{equation*}
a cone decomposition of $E' \oplus \tilde{E}'$ with components
\begin{equation*}
(G[n'_1],G[n'_2],\dots,G[n'_l])
\end{equation*}
and a cone decomposition of $E'' \oplus \tilde{E}''$ with components
\begin{equation*}
(G[n''_1],G[n''_2],\dots,G[n''_m])
\end{equation*}
so that
\begin{equation*}
\sum_{i=1}^k e^{n_it} < \delta_t(E,F) + \varepsilon, \quad \sum_{j=1}^l e^{n'_jt} < \delta_t(G,E') + \varepsilon \quad\text{and}\quad \sum_{j=1}^m e^{n''_jt} < \delta_t(G,E'') + \varepsilon.
\end{equation*}

Using the exact triangle $E' \to E \to E'' \to E'[1]$, we can further decompose the previously chosen cone decomposition of $F \oplus \tilde{F}$.
Then we obtain a new cone decomposition of $F \oplus \tilde{F}$ with components
\begin{equation*}
(E'[n_1],E''[n_1],\dots,E'[n_k],E''[n_k])
\end{equation*}
This implies that $F \oplus \tilde{F}'$, where $\tilde{F}' = \tilde{F} \oplus \tilde{E}'[n_1] \oplus \cdots \oplus \tilde{E}'[n_k] \oplus \tilde{E}''[n_1] \oplus \cdots \oplus \tilde{E}''[n_k]$, admits a cone decomposition with components
\begin{equation*}
((E' \oplus \tilde{E}')[n_1],(E'' \oplus \tilde{E}'')[n_1],\dots,(E' \oplus \tilde{E}')[n_k],(E'' \oplus \tilde{E}'')[n_k])
\end{equation*}
and thus one with components
\begin{gather*}
(G[n_1+n'_1],\dots,G[n_1+n'_l],G[n_1+n''_1],\dots,G[n_1+n''_m],\dots,\\
G[n_k+n'_1],\dots,G[n_k+n'_l],G[n_k+n''_1],\dots,G[n_k+n''_m]).
\end{gather*}
Therefore, it follows that
\begin{align*}
\delta_t(G,F)
&\leq \sum_{i=1}^k \sum_{j=1}^l e^{(n_i+n'_j)t} + \sum_{i=1}^k \sum_{j=1}^m e^{(n_i+n''_j)t}\\
&= \left(\sum_{j=1}^l e^{n'_jt} + \sum_{j=1}^m e^{n''_jt}\right)\sum_{i=1}^k e^{n_it}\\
&< (\delta_t(G,E') + \delta_t(G,E'') + 2\varepsilon)(\delta_t(E,F) + \varepsilon) .
\end{align*}
Since $\varepsilon > 0$ is arbitrary, this completes the proof.
\end{proof}

\begin{cor}\label{cor3.3}
Let $E$ (resp. $F$) be an object of a triangulated category $\mathcal{D}$ (resp. $\mathcal{C}$) and $\Phi : \mathcal{C} \to \mathcal{D}$ be an exact functor with right adjoint functor $R$.
Then
\begin{equation*}
\delta_t(E,\Phi F) \leq (1 + \delta_t(E,T_\Phi E[-1])) \delta_t(RE,F)
\end{equation*}
where $T_\Phi = \mathrm{Cone}(\Phi R \to \mathrm{Id}_\mathcal{D})$. 
\end{cor}

\begin{proof}
By definition, there is an exact triangle $T_\Phi E[-1] \to \Phi RE \to E \to T_\Phi E$.
Hence ,we see that
\begin{align*}
(1 + \delta_t(E,T_\Phi E[-1])) \delta_t(RE,F)
&\geq (1 + \delta_t(E,T_\Phi E[-1])) \delta_t(\Phi RE,\Phi F)\\
&\geq (\delta_t(E,E) + \delta_t(E,T_\Phi E[-1])) \delta_t(\Phi RE,\Phi F)\\
&\geq \delta_t(E,\Phi F)
\end{align*}
where the first inequality follows from Lemma \ref{lem2.2} (3) and the last inequality follows from Lemma \ref{lem3.2}.
\end{proof}

Recall that a sequence $\{a_n\}_{n=1}^\infty$ is called {\em submultiplicative} if $a_{n+m} \leq a_na_m$ for all $n,m$.
If $a_n > 0$ for all $n$, then the sequence $\{b_n = \log a_n\}_{n=1}^\infty$ is {\em subadditive}, i.e., $b_{n+m} \leq b_n + b_m$ for all $n,m$.
Therefore the sequence $\{\frac{1}{n} \log a_n\}_{n=1}^\infty$ converges and
\begin{equation*}
\lim_{n \to \infty} \frac{1}{n} \log a_n = \inf_{n \geq 1} \frac{1}{n} \log a_n
\end{equation*}
by the {\em Fekete's subadditive lemma}.

\begin{lem}\label{lem3.4}
Let $\{a_n\}_{n=1}^\infty$ be a submultiplicative sequence of positive real numbers.
Then the sequence $\{\frac{1}{n} \log (1 + \sum_{i=1}^n a_i)\}_{n=1}^\infty$ converges and
\begin{equation*}
\lim_{n \to \infty} \frac{1}{n} \log \left(1 + \sum_{i=1}^n a_i\right) \leq
\begin{cases}
0 & (\text{if } \lim \frac{1}{n} \log a_n \leq 0),\\
\displaystyle \lim_{n \to \infty} \frac{1}{n} \log a_n & (\text{if } \lim \frac{1}{n} \log a_n \geq 0).
\end{cases}
\end{equation*}
\end{lem}

\begin{proof}
The sequence $\{1 + \sum_{i=1}^n a_i\}_{n=1}^\infty$ is submultiplicative.
Indeed,
\begin{equation*}
1 + \sum_{i=1}^{n+m} a_i \leq 1 + \sum_{i=1}^n a_i + \sum_{i=n+1}^{n+m} a_i \leq 1 + \sum_{i=1}^n a_i + a_n \sum_{i=1}^m a_i \leq \left(1 + \sum_{i=1}^n a_i\right) \left(1 + \sum_{i=1}^m a_i\right).
\end{equation*}
Thus the sequence $\{\frac{1}{n} \log (1 + \sum_{i=1}^n a_i)\}_{n=1}^\infty$ converges.

Now let
\begin{equation*}
L = \lim_{n \to \infty} \frac{1}{n} \log a_n = \inf_{n \geq 1} \frac{1}{n} \log a_n.
\end{equation*}

(Case 1)
Suppose $L < 0$.
By taking a subsequence if necessary, we can assume that $a_n < 1$ for all $n$.
Then
\begin{equation*}
\lim_{n \to \infty} \frac{1}{n} \log \left(1 + \sum_{i=1}^n a_i\right) \leq \lim_{n \to \infty} \frac{1}{n} \log (1 + n) = 0.
\end{equation*}

(Case 2)
Suppose $L > 0$.
By definition, we have $a_n \geq e^{Ln}$ and hence $a_n \to \infty$ as $n \to \infty$.
Therefore we can take an increasing subsequence which we denote by the same notation $\{a_n\}_{n=1}^\infty$.
Then
\begin{equation*}
\lim_{n \to \infty} \frac{1}{n} \log \left(1 + \sum_{i=1}^n a_i\right) \leq \lim_{n \to \infty} \frac{1}{n} \log (1 + na_n) = \lim_{n \to \infty} \frac{1}{n} \log a_n.
\end{equation*}

(Case 3)
Suppose $L = 0$.
By definition, we have $a_n \geq 1$ for all $n$.
If the sequence $\{a_n\}$ is bounded, then we can argue as in (Case 1).
If the sequence $\{a_n\}$ is unbounded, then we can take an increasing subsequence.
A similar argument as in (Case 2) then shows the assertion.
\end{proof}

\begin{proof}[Proof of Theorem \ref{thm1.6}]
Let $G,G'$ be split-generators of $\mathcal{D}$ such that $RG,RG'$ are split-generators of $\mathcal{C}$.
Applying $T_S^{n-1}$ to the exact triangle defining the twist functor $T_S$, we obtain
\begin{equation*}
T_S^{n-1}G \to T_S^nG \to T_S^{n-1}SRG[1] \to T_S^{n-1}G[1]
\end{equation*}
and therefore
\begin{align*}
\delta_t(G',T_S^nG)
&\leq \delta_t(G',T_S^{n-1}G) + \delta_t(G',T_S^{n-1}SRG[1])\\
&= \delta_t(G',T_S^{n-1}G) + \delta_t(G',S(C_S[2])^{n-1}RG[1])\\
&\leq \delta_t(G',T_S^{n-1}G) + (1 + \delta_t(G',T_SG'[-1])) \delta_t(RG',(C_S[2])^{n-1}RG[1])
\end{align*}
where the first inequality follows from Lemma \ref{lem3.1}, the second equality follows from Lemma \ref{lem2.9} (1) and the last inequality follows from Corollary \ref{cor3.3}.
Continuing this process, we get
\begin{align*}
\delta_t(G',T_S^nG)
&\leq \delta_t(G',G) + (1 + \delta_t(G',T_SG'[-1])) \sum_{i=0}^{n-1} \delta_t(RG',(C_S[2])^iRG[1])\\
&\leq M_t\left(1 + \sum_{i=0}^{n-1} \delta_t(RG',(C_S[2])^iRG[1])\right)
\end{align*}
where $M_t = \max\{\delta_t(G',G),1 + \delta_t(G',T_SG'[-1])\}$ which is independent of $n$.

Consequently, we have
\begin{align*}
h_t(T_S)
&= \lim_{n \to \infty} \frac{1}{n} \log \delta_t(G',T_S^nG)\\
&\leq \lim_{n \to \infty} \frac{1}{n} \log \left(M_t\left(1 + \sum_{i=0}^{n-1} \delta_t(RG',(C_S[2])^iRG[1])\right)\right)\\
&= \lim_{n \to \infty} \frac{1}{n} \log \left(1 + \sum_{i=0}^{n-1} \delta_t(RG',(C_S[2])^iRG[1])\right)\\
&\leq
\begin{cases}
0 & (\text{if } h_t(C_S[2]) \leq 0),\\
\displaystyle \lim_{n \to \infty} \frac{1}{n} \log \delta_t(RG',(C_S[2])^nRG[1]) = h_t(C_S[2]) & (\text{if } h_t(C_S[2]) \geq 0).
\end{cases}
\end{align*}
Here Lemma \ref{lem3.4} is used in the last inequality.
\end{proof}

\section{Application to Gromov--Yomdin type formula}\label{sec4}

Let $\mathcal{D}$ be a triangulated category.
The {\em Grothendieck group} $K(\mathcal{D})$ of $\mathcal{D}$ is defined as
\begin{equation*}
K(\mathcal{D}) = \mathbb{Z}\langle\mathrm{Ob}(\mathcal{D})\rangle/\langle E-F+G \,|\, E \to F \to G \to E[1]\rangle.
\end{equation*}
Now suppose $\mathcal{D}$ is of {\em finite type}, i.e.,
\begin{equation*}
\dim\mathrm{Hom}_\mathcal{D}^*(E,F) = \sum_{i \in \mathbb{Z}} \dim\mathrm{Hom}_\mathcal{D}(E,F[i]) < \infty
\end{equation*}
for any $E,F \in \mathrm{Ob}(\mathcal{D})$.
Then the {\em Euler form} $\chi : K(\mathcal{D}) \times K(\mathcal{D}) \to \mathbb{Z}$ given by
\begin{equation*}
\chi([E],[F]) = \sum_{i \in \mathbb{Z}} (-1)^i \dim\mathrm{Hom}_\mathcal{D}(E,F[i])
\end{equation*}
is well-defined.
We define the {\em numerical Grothendieck group} $\mathcal{N}(\mathcal{D})$ as
\begin{equation*}
\mathcal{N}(\mathcal{D}) = K(\mathcal{D})/\langle [E] \in K(\mathcal{D}) \,|\, \chi([E],-) = 0 \rangle.
\end{equation*}
In this section, we only consider triangulated categories whose numerical Grothendieck groups are of finite rank.

Now let $\Phi : \mathcal{C} \to \mathcal{D}$ be an exact functor between triangulated categories admitting a right adjoint functor.
Then the induced homomorphism $[\Phi] : K(\mathcal{C}) \to K(\mathcal{D})$ descends to the numerical Grothendieck groups.
Tensoring with $\mathbb{C}$, we obtain the homomorphism
\begin{equation*}
[\Phi] : \mathcal{N}(\mathcal{C}) \otimes_\mathbb{Z} \mathbb{C} \to \mathcal{N}(\mathcal{D}) \otimes_\mathbb{Z} \mathbb{C}
\end{equation*}
which we denote by the same notation.

\begin{proof}[Proof of Proposition \ref{prop1.9}]
By Theorem \ref{thm1.6} and the assumption,
\begin{equation*}
h_0(T_S) \leq h_0(C_S) \leq \log\rho([C_S]).
\end{equation*}
Now take an element $v \in \mathcal{N}(\mathcal{C}) \otimes_\mathbb{Z} \mathbb{C}$ such that $[S]v \neq 0$ and $[C_S]v = \lambda v$ where $|\lambda| = \rho([C_S])$ which exists by our assumption.
Then, by Lemma \ref{lem2.9} (1),
\begin{equation*}
[T_S][S]v = [T_SS]v = [SC_S]v = [S][C_S]v = \lambda [S]v.
\end{equation*}
This shows that $\rho([C_S]) \leq \rho([T_S])$ and therefore
\begin{equation*}
h_0(T_S) \leq \log\rho([T_S])
\end{equation*}
as desired.
\end{proof}

Recall that a dg algebra $A$ is called {\em smooth} if $A$ is perfect as a dg bimodule over itself and {\em proper} if $A$ has the finite dimensional total cohomology.

\begin{thm}[\cite{KST}, Theorem 2.13]\label{thm4.1}
Let $\mathcal{D}$ be the perfect derived category of a smooth proper dg algebra and $\Phi : \mathcal{D} \to \mathcal{D}$ be an exact endofunctor admitting a right adjoint functor.
Then
\begin{equation*}
h_0(\Phi) \geq \log\rho([\Phi]).
\end{equation*}
\end{thm}

\begin{proof}[Proof of Corollary \ref{cor1.10}]
It follows directly from Proposition \ref{prop1.9} and Theorem \ref{thm4.1}.
\end{proof}

\section{Examples}\label{sec5}

\subsection{Spherical objects}

Spherical objects and twists were introduced by Seidel--Thomas \cite{ST} as the mirror counterparts of Lagrangian spheres and the Dehn twists along them.

\begin{dfn}[\cite{ST}, Definition 2.9]
An object $E$ of a triangulated category $\mathcal{D}$ of finite type is called a {\em $d$-spherical object} ($d > 0$) if it satisfies the following conditions:
\begin{itemize}
\item[(1)] $\mathrm{Hom}_\mathcal{D}^*(E,E) \cong \mathbb{K}[h]/(h^2)$ where $\deg h = d$.
\item[(2)] There is a functorial isomorphism $\mathrm{Hom}_\mathcal{D}^*(E,-) \cong \mathrm{Hom}_\mathcal{D}^*(-,E[d])^\vee$.
\end{itemize}
\end{dfn}

\begin{thm}[\cite{ST}, Proposition 2.10]
Let $E$ be a $d$-spherical object of a triangulated category $\mathcal{D}$ of finite type.
Then there is an exact autoequivalence $T^\mathbb{S}_E$ of $\mathcal{D}$, called the {\em spherical twist} along $E$, which acts on objects by
\begin{equation*}
T^\mathbb{S}_EF = \mathrm{Cone}(\mathrm{Hom}_\mathcal{D}^*(E,F) \otimes E \xrightarrow{\mathrm{ev}} F).
\end{equation*}
\end{thm}

Let $E$ be a $d$-spherical object of a triangulated category $\mathcal{D}$ of finite type.
Let us regard $\mathbb{K}$ as a dg algebra concentrated in degree 0 (so with zero differential).
Denote by $D_{fd}(\mathbb{K})$ the derived category of dg $\mathbb{K}$-modules with finite dimensional cohomology (which is equivalent to the bounded derived category of finite dimensional vector spaces over $\mathbb{K}$).
Then the functor
\begin{equation*}
S = - \otimes_\mathbb{K} E : D_{fd}(\mathbb{K}) \to \mathcal{D}
\end{equation*}
is a spherical functor and the twist and cotwist functors along it can be written as
\begin{equation*}
 T_S \cong T^\mathbb{S}_E \text{ and } C_S \cong [-1-d].
\end{equation*}
Clearly the essential image of the right adjoint functor of $S$ contains a split-generator of $D_{fd}(\mathbb{K})$ (or alternatively, we can apply Corollary \ref{cor5.13}).
Thus, since
\begin{equation*}
h_t(C_S[2]) = h_t([1-d]) = (1-d)t,
\end{equation*}
we recover Ouchi's result \cite[Theorem 3.1]{Ouc} (see Theorem \ref{thm1.3}) from Theorems \ref{thm1.6} and \ref{thm1.7}.

\begin{prop}
Let $E$ be a $d$-spherical object of a triangulated category $\mathcal{D}$ of finite type admitting a split-generator.
Then
\begin{equation*}
(1-d)t \leq h_t(T^\mathbb{S}_E) \leq
\begin{cases}
0 & (\text{if } t \geq 0),\\
(1-d)t & (\text{if } t \leq 0).
\end{cases}
\end{equation*}
If moreover $E^\perp \neq 0$ then $h_t(T^\mathbb{S}_E) = 0$ for all $t \geq 0$.
\end{prop}

\subsection{$\mathbb{P}$-objects}

Similarly to spherical objects and twists, $\mathbb{P}$-objects and twists, introduced by Huybrechts--Thomas \cite{HT}, can be considered as the mirror counterparts of Lagrangian projective spaces and the Dehn twists along them.

\begin{dfn}[\cite{HT}, Definition 1.1]
An object $E$ of a triangulated category $\mathcal{D}$ of finite type is called a {\em $\mathbb{P}^d$-object} ($d > 0$) if it satisfies the following conditions:
\begin{itemize}
\item[(1)] $\mathrm{Hom}_\mathcal{D}^*(E,E) \cong \mathbb{K}[h]/(h^{d+1})$ where $\deg h = 2$.
\item[(2)] There is a functorial isomorphism $\mathrm{Hom}_\mathcal{D}^*(E,-) \cong \mathrm{Hom}_\mathcal{D}^*(-,E[2d])^\vee$.
\end{itemize}
\end{dfn}

\begin{thm}[\cite{HT}, Proposition 2.6]
Let $E$ be a $\mathbb{P}^d$-object of a triangulated category $\mathcal{D}$ of finite type.
Then there is an exact autoequivalence $T^\mathbb{P}_E$ of $\mathcal{D}$, called the {\em $\mathbb{P}$-twist} along $E$, which acts on objects by
\begin{equation*}
T^\mathbb{P}_EF = \mathrm{Cone}(\mathrm{Cone}(\mathrm{Hom}_\mathcal{D}^*(E,F) \otimes E[-2] \xrightarrow{h^* \otimes \mathrm{id} - \mathrm{id} \otimes h} \mathrm{Hom}_\mathcal{D}^*(E,F) \otimes E) \xrightarrow{\mathrm{ev}} F)
\end{equation*}
where $h^* : \mathrm{Hom}_\mathcal{D}^*(E,F) \to \mathrm{Hom}_\mathcal{D}^*(E,F)[2]$ is the homomorphism given by precomposing with $h \in \mathrm{Hom}_\mathcal{D}(E,E[2])$.
\end{thm}

\begin{rmk}
An object $E$ is a $\mathbb{P}^1$-object if and only if it is a 2-spherical object.
In that case, there is an isomorphism $T^\mathbb{P}_E \cong (T^\mathbb{S}_E)^2$ \cite[Proposition 2.9]{HT}.
\end{rmk}

Let $E$ be a $\mathbb{P}^d$-object of a triangulated category $\mathcal{D}$ of finite type.
Consider the dg algebra $\mathbb{K}[h]$ with $\deg h = 2$ (so with zero differential).
Denote by $D_{fd}(\mathbb{K}[h])$ the derived category of dg $\mathbb{K}[h]$-modules with finite dimensional cohomology.
Then the functor
\begin{equation*}
S = - \otimes_{\mathbb{K}[h]} E : D_{fd}(\mathbb{K}[h]) \to \mathcal{D}
\end{equation*}
is a spherical functor and the twist and cotwist functors along it can be written as
\begin{equation*}
 T_S \cong T^\mathbb{P}_E \text{ and } C_S \cong [-2-2d]
\end{equation*}
\cite[Proposition 4.2]{Seg},\cite[Corollary 2.9]{HK}.
Clearly the essential image of the right adjoint functor of $S$ contains a split-generator of $D_{fd}(\mathbb{K}[h])$ (or alternatively, we can apply Corollary \ref{cor5.13}).
Thus, since
\begin{equation*}
h_t(C_S[2]) = h_t([-2d]) = -2dt,
\end{equation*}
we recover Fan's result \cite[Theorem 3.1]{Fan2} (see Theorem \ref{thm1.5}) from Theorems \ref{thm1.6} and \ref{thm1.7}.

\begin{prop}
Let $E$ be a $\mathbb{P}^d$-object of a triangulated category $\mathcal{D}$ of finite type admitting a split-generator.
Then
\begin{equation*}
-2dt \leq h_t(T^\mathbb{P}_E) \leq
\begin{cases}
0 & (\text{if } t \geq 0),\\
-2dt & (\text{if } t \leq 0).
\end{cases}
\end{equation*}
If moreover $E^\perp \neq 0$ then $h_t(T^\mathbb{P}_E) = 0$ for all $t \geq 0$.
\end{prop}

\subsection{Orthogonally spherical objects}

Spherical functors which can be expressed as Fourier--Mukai transforms have been studied by many authors: Horja \cite{Hor}, Toda \cite{Tod}, Anno--Logvinenko \cite{AL1} and many others.
In this section, we focus on orthogonally spherical objects and the twists along them introduced by Anno--Logvinenko \cite{AL1}.
In what follows, we only consider equidimensional schemes for simplicity and all functors appearing in this section are always assumed to be derived although we will not explicitly indicate that by the notation.

Let $X,Z$ be separated Gorenstein schemes of finite type over $\mathbb{C}$.
For a closed point $p \in Z$, denote the corresponding inclusion by $\iota_{p \times X} : X \to Z \times X$.
For a perfect object $\mathcal{E} \in \mathrm{Ob}(D^b(Z \times X))$, define its {\em fiber} at a closed point $p \in Z$ by $\mathcal{E}_p = \iota_{p \times X}^*\mathcal{E}$.
Let us also write the projections by $\pi_X : Z \times X \to X$ and $\pi_Z : Z \times X \to Z$.
We define an object $\mathcal{L}_\mathcal{E} \in \mathrm{Ob}(D^b(Z))$ to be the cone of the composition of the morphisms
\begin{equation*}
\mathcal{O}_Z \to \pi_{Z*}\mathcal{O}_{Z \times X} \to \pi_{Z*}\mathcal{H}om_{Z \times X}(\mathcal{E},\mathcal{E}) \to \pi_{Z*}\mathcal{H}om_{Z \times X}(\pi_X^*\pi_{X*}\mathcal{E},\mathcal{E})
\end{equation*}
where the first morphism is induced by the adjunction unit $\mathrm{Id}_{D^b(Z)} \to \pi_{Z*}\pi_Z^*$, the second morphism is induced by the adjunction unit $\mathrm{Id}_{D^b(Z \times X)} \to \mathcal{H}om_{Z \times X}(\mathcal{E},- \otimes \mathcal{E})$ and the third morphism is induced by the adjunction counit $\pi_X^*\pi_{X*} \to \mathrm{Id}_{D^b(X)}$.

\begin{dfn}[\cite{AL1}, Definition 3.4]\label{dfn5.8}
Let $X,Z$ be separated Gorenstein schemes of finite type over $\mathbb{C}$ with $\dim X > \dim Z$.
A perfect object $\mathcal{E} \in \mathrm{Ob}(D^b(Z \times X))$ is called an {\em orthogonally spherical object} if it satisfies the following conditions:
\begin{itemize}
\item[(1)] $\mathrm{Hom}_{D^b(X)}^*(\pi_{X*}\mathcal{E},\mathcal{E}_p) \cong \mathbb{C}[h]/(h^2)$ where $\deg h = \dim X - \dim Z$ for every closed point $p \in Z$.
\item[(2)] $\mathcal{E}^\vee \otimes \pi_X^!\mathcal{O}_X \cong \mathcal{E}^\vee \otimes \pi_Z^!\mathcal{L}_\mathcal{E}$.
\item[(3)] $\mathrm{Hom}_{D^b(X)}^*(\mathcal{E}_p,\mathcal{E}_q) = 0$ for every pair of distinct closed points $p,q \in Z$.
\end{itemize}
\end{dfn}

\begin{rmk}
The above definition is stronger than the original definition \cite[Definition 3.4]{AL1} where $\mathcal{E} \in \mathrm{Ob}(D^b(Z \times X))$ is called {\em spherical} if the Fourier--Mukai transform $\Phi_\mathcal{E}^{Z \to X} : D^b(Z) \to D^b(X)$ is a spherical functor and {\em orthogonal} if it satisfies Definition \ref{dfn5.8} (3).
The dimension restriction (i.e., $\dim X > \dim Z$) is included just for the sake of simplicity.
\end{rmk}

\begin{thm}[\cite{AL1}, Theorem 3.2 and Proposition 3.7]\label{thm5.10}
Let $X,Z$ be separated Gorenstein schemes of finite type over $\mathbb{C}$ and $\mathcal{E} \in \mathrm{Ob}(D^b(Z \times X))$ be an orthogonally spherical object.
Then the Fourier--Mukai transform $\Phi_\mathcal{E}^{Z \to X} : D^b(Z) \to D^b(X)$ with kernel $\mathcal{E}$ is a spherical functor.
Moreover, $\mathcal{L}_\mathcal{E}[\dim X - \dim Z]$ is isomorphic to an invertible sheaf on $Z$ and the cotwist functor along the spherical functor $\Phi_\mathcal{E}^{Z \to X}$ is isomorphic to $- \otimes \mathcal{L}_\mathcal{E}[-1] : D^b(Z) \to D^b(Z)$.
\end{thm}

For an orthogonally spherical object $\mathcal{E} \in \mathrm{Ob}(D^b(Z \times X))$, we denote the twist functor along the spherical functor $\Phi_\mathcal{E}^{Z \to X}$ by $T^\mathbb{O}_\mathcal{E} : D^b(X) \to D^b(X)$.

We can apply Theorems \ref{thm1.6} and \ref{thm1.7} to the following situation.

\begin{prop}\label{prop5.11}
Let $X,Z$ be smooth projective varieties over $\mathbb{C}$ and $\mathcal{E} \in \mathrm{Ob}(D^b(Z \times X))$ be an orthogonally spherical object.
Then
\begin{equation*}
(1-d)t \leq h_t(T^\mathbb{O}_\mathcal{E}) \leq
\begin{cases}
0 & (\text{if } t \geq 0),\\
(1-d)t & (\text{if } t \leq 0)
\end{cases}
\end{equation*}
where $d = \dim X - \dim Z > 0$.
If moreover $\bigcap_{p \in Z} \mathcal{E}_p^\perp \neq 0$ then $h_t(T^\mathbb{O}_\mathcal{E}) = 0$ for all $t \geq 0$.
\end{prop}

To show Proposition \ref{prop5.11}, let us first prove a general lemma which helps to check the technical condition in Theorems \ref{thm1.6} and \ref{thm1.7}.

Let $S : \mathcal{C} \to \mathcal{D}$ be a spherical functor with right adjoint functor $R$.
Fix a split-generator $G \in \mathrm{Ob}(\mathcal{C})$ and set $X_1 = RSG \in \mathrm{Ob}(\mathcal{C})$.
It sits in the exact triangle
\begin{equation*}
G \overset{\eta_G}{\to} RSG \overset{\phi_1}{\to} C_SG[1] \to G[1]
\end{equation*}
defining the cotwist functor $C_S$.
Then we inductively define a sequence $\{X_n\}_{n=1}^\infty$ of objects of $\mathcal{C}$ to fit into the exact triangles
\begin{equation*}
\begin{tikzcd}
X_{n-1}[-1] \ar[r,"\phi_{n-1}{[-1]}"] \ar[d,equal] & C_S^{n-1}G \ar[d,"\eta_{C_S^{n-1}G}"] \ar[r] & G \ar[r] \ar[d] & X_{n-1} \ar[d,equal]\\
X_{n-1}[-1] \ar[r] & RSC_S^{n-1}G \ar[r] \ar[d] & X_n \ar[r] \ar[d,"\phi_n"] & X_{n-1}\\
& C_S^nG[1] \ar[r,equal] \ar[d] & C_S^nG[1] \ar[d] &\\
& C_S^{n-1}G[1] \ar[r] & G[1]. &
\end{tikzcd}
\end{equation*}

\begin{lem}
Let $S : \mathcal{C} \to \mathcal{D}$ be a spherical functor with right adjoint functor $R$ and $G$ be a split-generator of $\mathcal{C}$.
Assume that there is an integer $n > 0$ such that $\mathrm{Hom}_\mathcal{C}(C_S^nG,G) = 0$.
Then $RS(G \oplus C_SG \oplus \cdots \oplus C_S^{n-1}G)$ is a split-generator of $\mathcal{C}$.
\end{lem}

\begin{proof}
The assumption implies that the exact triangle
\begin{equation*}
G \to X_n \to C_S^nG[1] \to G[1]
\end{equation*}
splits and thus $X_n \cong G \oplus C_S^nG[1]$.
On the other hand, $X_n$ can be split-generated by $RS(G \oplus C_SG \oplus \cdots \oplus C_S^{n-1}G)$ by construction.
\end{proof}

\begin{cor}\label{cor5.13}
Let $S : \mathcal{C} \to \mathcal{D}$ be a spherical functor with right adjoint functor $R$.
If there exist a split-generator $G$ of $\mathcal{C}$ and an integer $n > 0$ such that $\mathrm{Hom}_\mathcal{C}(C_S^nG,G) = 0$, then the essential image of $RS$ (in particular, that of $R$) contains a split-generator of $\mathcal{C}$.
\end{cor}

\begin{proof}[Proof of Proposition \ref{prop5.11}]
Let $S = \Phi_\mathcal{E}^{Z \to X} : D^b(Z) \to D^b(X)$ be the Fourier--Mukai transform with kernel $\mathcal{E}$ and $R$ be its right adjoint functor.
Let $\mathcal{G} \in \mathrm{Ob}(\mathrm{Coh}(Z))$ be a split-generator of $D^b(Z)$.
By Theorem \ref{thm5.10}, there is an invertible sheaf $\mathcal{L}'_\mathcal{E}$ on $Z$ which is isomorhpic to $\mathcal{L}_\mathcal{E}[d]$.
Then we have
\begin{align*}
\mathrm{Hom}_{D^b(Z)}(C_S^n\mathcal{G},\mathcal{G})
&\cong \mathrm{Hom}_{D^b(Z)}(\mathcal{G} \otimes \mathcal{L}_\mathcal{E}^{\otimes n}[-n],\mathcal{G})\\
&\cong \mathrm{Hom}_{D^b(Z)}(\mathcal{G} \otimes \mathcal{L}_\mathcal{E}'^{\otimes n}[-n(1+d)],\mathcal{G})\\
&\cong \mathrm{Hom}_{D^b(Z)}(\mathcal{G} \otimes \mathcal{L}_\mathcal{E}'^{\otimes n},\mathcal{G}[n(1+d)])\\
&\cong \mathrm{Ext}_Z^{n(1+d)}(\mathcal{G} \otimes \mathcal{L}_\mathcal{E}'^{\otimes n},\mathcal{G})\\
&= 0
\end{align*}
for every $n > 0$ satisfying $n(1+d) > \dim Z$ (e.g. take any $n \geq \dim Z$).
By Corollary \ref{cor5.13}, this implies that the essential image of $R$ contains a split-generator of $\mathcal{C}$.
The former assertion then follows from Theorems \ref{thm1.6} and \ref{thm1.7} (1) as
\begin{equation*}
h_t(C_S[2]) = h_t(- \otimes \mathcal{L}_\mathcal{E}[1]) = h_t(- \otimes \mathcal{L}'_\mathcal{E}[1-d]) = (1-d)t
\end{equation*}
by \cite[Lemma 2.14]{DHKK}.

Let us next prove the latter assertion.
More precisely, we shall show that
\begin{equation*}
\bigcap_{p \in Z} \mathcal{E}_p^\perp \subset \mathrm{Ker}\, SR
\end{equation*}
then the assertion follows from Theorem \ref{thm1.7} (2).
Consider the following set of projections
\begin{equation*}
\begin{tikzcd}
& Z &\\
X \times Z \ar[d,"\pi_X",swap] \ar[ur,"\pi_Z"] & & Z \times X \ar[d,"\pi_X"] \ar[ul,"\pi_Z",swap]\\
X & X \times Z \times X \ar[ul,"\pi_{12}",swap] \ar[ur,"\pi_{23}"] \ar[d,"\pi_{13}"] & X\\
& X \times X. \ar[ul,"\pi_1"] \ar[ur,"\pi_2",swap] &
\end{tikzcd}
\end{equation*}
The functor $SR : D^b(X) \to D^b(X)$ is isomorphic to the Fourier--Mukai transform with kernel $\pi_{13*}(\pi_{12}^*(\mathcal{E}^\vee \otimes \pi_X^!\mathcal{O}_X) \otimes \pi_{23}^*\mathcal{E}) \in \mathrm{Ob}(D^b(X \times X))$.
Therefore, for an object $\mathcal{F} \in \mathrm{Ob}(D^b(X))$, we have
\begin{align*}
SR\mathcal{F}
&\cong \pi_{2*}(\pi_{13*}(\pi_{12}^*(\mathcal{E}^\vee \otimes \pi_X^!\mathcal{O}_X) \otimes \pi_{23}^*\mathcal{E}) \otimes \pi_1^*\mathcal{F})\\
&\cong \pi_{2*}\pi_{13*}(\pi_{12}^*(\mathcal{E}^\vee \otimes \pi_X^!\mathcal{O}_X) \otimes \pi_{23}^*\mathcal{E} \otimes \pi_{13}^*\pi_1^*\mathcal{F})\\
&\cong \pi_{X*}\pi_{23*}(\pi_{12}^*(\mathcal{E}^\vee \otimes \pi_X^!\mathcal{O}_X) \otimes \pi_{23}^*\mathcal{E} \otimes \pi_{12}^*\pi_X^*\mathcal{F})\\
&\cong \pi_{X*}\pi_{23*}(\pi_{23}^*\mathcal{E} \otimes \pi_{12}^*(\mathcal{E}^\vee \otimes \pi_X^!\mathcal{O}_X \otimes \pi_X^*\mathcal{F}))\\
&\cong \pi_{X*}(\mathcal{E} \otimes \pi_{23*}\pi_{12}^*(\mathcal{E}^\vee \otimes \pi_X^!\mathcal{O}_X \otimes \pi_X^*\mathcal{F}))\\
&\cong \pi_{X*}(\mathcal{E} \otimes \pi_{23*}\pi_{12}^*(\mathcal{E}^\vee \otimes \pi_X^!\mathcal{F}))\\
&\cong \pi_{X*}(\mathcal{E} \otimes \pi_Z^*\pi_{Z*}(\mathcal{E}^\vee \otimes \pi_X^!\mathcal{F}))\\
&\cong \pi_{X*}(\mathcal{E} \otimes \pi_Z^*\pi_{Z*}\mathcal{H}om_{Z \times X}(\mathcal{E},\pi_X^!\mathcal{F}))
\end{align*}
where the second and fifth isomorphisms follow from the projection formula and the seventh isomorphism follows from the base change theorem.
Now suppose $\mathcal{F} \in \mathcal{E}_p^\perp$ for every closed point $p \in Z$.
Let us show that
\begin{equation*}
\pi_{Z*}\mathcal{H}om_{Z \times X}(\mathcal{E},\pi_X^!\mathcal{F}) \cong 0.
\end{equation*}
Consider the following set of morphisms
\begin{equation*}
\begin{tikzcd}
X \ar[r,"\iota_{p \times X}"] \ar[d,"\pi_\mathbb{C}",swap] & Z \times X \ar[d,"\pi_Z"]\\
\mathrm{Spec}\,\mathbb{C} \ar[r,"\iota_p",swap] & Z
\end{tikzcd}
\end{equation*}
where $\iota_{p \times X}$ and $\iota_p$ are the inclusions corresponding to a closed point $p \in Z$.
Denote by $S_X = - \otimes \omega_X[\dim X] : D^b(X) \to D^b(X)$ and $S_{Z \times X} = - \otimes \pi_X^*\omega_X \otimes \pi_Z^*\omega_Z[\dim X + \dim Z] : D^b(Z \times X) \to D^b(Z \times X)$ the Serre functors.
Since both $X$ and $Z$ are smooth projective varieties (and so is $Z \times X$), we see that
\begin{align*}
\pi_X^!\mathcal{F}
&\cong S_{Z \times X}\pi_X^*S_X^{-1}\mathcal{F}\\
&= S_{Z \times X}\pi_X^*(\mathcal{F} \otimes \omega_X^\vee[-\dim X])\\
&\cong S_{Z \times X}(\pi_X^*\mathcal{F} \otimes \pi_X^*\omega_X^\vee[-\dim X])\\
&= \pi_X^*\mathcal{F} \otimes \pi_X^*\omega_X^\vee \otimes \pi_X^*\omega_X \otimes \pi_Z^*\omega_Z[\dim Z]\\
&\cong \pi_X^*\mathcal{F} \otimes \pi_Z^*\omega_Z[\dim Z]
\end{align*}
and thus
\begin{equation*}
\iota_{p \times X}^*\pi_X^!\mathcal{F} \cong \iota_{p \times X}^*\pi_X^*\mathcal{F} \otimes \iota_{p \times X}^*\pi_Z^*\omega_Z[\dim Z] \cong \mathcal{F}[\dim Z].
\end{equation*}
Consequently, we have
\begin{align*}
\iota_p^*\pi_{Z*}\mathcal{H}om_{Z \times X}(\mathcal{E},\pi_X^!\mathcal{F})
&\cong \pi_{\mathbb{C}*}\iota_{p \times X}^*\mathcal{H}om_{Z \times X}(\mathcal{E},\pi_X^!\mathcal{F})\\
&\cong \pi_{\mathbb{C}*}\mathcal{H}om_X(\iota_{p \times X}^*\mathcal{E},\iota_{p \times X}^*\pi_X^!\mathcal{F})\\
&\cong \mathrm{Hom}_{D^b(X)}^*(\mathcal{E}_p,\mathcal{F}[\dim Z])\\
&= 0
\end{align*}
where the first isomorphism follows from the base change theorem and the last equality follows from the assumption that $\mathcal{F} \in \mathcal{E}_p^\perp = \{\mathcal{F} \in \mathrm{Ob}(D^b(X)) \,|\, \mathrm{Hom}_{D^b(X)}^*(\mathcal{E}_p,\mathcal{F}) = 0\}$.
As this holds for every closed point $p \in Z$, we conclude that $\pi_{Z*}\mathcal{H}om_{Z \times X}(\mathcal{E},\pi_X^!\mathcal{F}) \cong 0$ (see \cite[Lemma 2.8]{AL1}).
\end{proof}

We can also prove the Gromov--Yomdin type formula for the twist along an orthogonally spherical object using Corollary \ref{cor1.10}.

\begin{prop}
Let $X,Z$ be smooth projective varieties over $\mathbb{C}$ and $\mathcal{E} \in \mathrm{Ob}(D^b(Z \times X))$ be an orthogonally spherical object.
Assume that there is a closed point $p \in Z$ such that $[\mathcal{E}_p] \neq 0$ in $\mathcal{N}(D^b(X))$.\footnote{
For instance, this condition holds if $\dim X - \dim Z$ is even by Definition \ref{dfn5.8} (1).
}
Then
\begin{equation*}
h_0(T^\mathbb{O}_\mathcal{E}) = \log\rho([T^\mathbb{O}_\mathcal{E}]) = 0.
\end{equation*}
\end{prop}

\begin{proof}
Let $S = \Phi_\mathcal{E}^{Z \to X} : D^b(Z) \to D^b(X)$ be the Fourier--Mukai transform with kernel $\mathcal{E}$.
Since $h_0(C_S) = \log\rho([C_S]) = 0$, the assertion follows from Corollary \ref{cor1.10} if we show that there exists an element $v \in \mathcal{N}(D^b(Z)) \otimes_\mathbb{Z} \mathbb{C}$ such that $[S]v \neq 0$ and $[C_S]v = \lambda v$ where $|\lambda| = \rho([C_S]) = 1$.

Take a closed point $p \in Z$ so that $[\mathcal{E}_p] \neq 0$ in $\mathcal{N}(D^b(X))$ and denote by $\mathcal{O}_p$ the skyscraper sheaf supported at $p$.
Then $S\mathcal{O}_p = \Phi_\mathcal{E}^{Z \to X}\mathcal{O}_p = \mathcal{E}_p$ and so $[S][\mathcal{O}_p] = [\mathcal{E}_p] \neq 0$.
On the other hand, choose an invertible sheaf $\mathcal{L}'_\mathcal{E}$ on $Z$ so that $\mathcal{L}_\mathcal{E}[d] \cong \mathcal{L}'_\mathcal{E}$ where $d = \dim X - \dim Z$.
Then $C_S\mathcal{O}_p \cong \mathcal{O}_p \otimes \mathcal{L}'_\mathcal{E}[-1-d] \cong \mathcal{O}_p[-1-d]$ and therefore $[C_S][\mathcal{O}_p] = \pm[\mathcal{O}_p]$.
This shows that we can take $v = [\mathcal{O}_p]$.
\end{proof}

\end{document}